\documentclass[reqno, a4paper]{amsart}

\usepackage[british]{babel}
\usepackage{mathrsfs}
\usepackage{amsmath}
\usepackage{amsthm}
\usepackage{amsrefs}
\usepackage{graphicx}
\usepackage{hyperref}

\usepackage[all]{xy}

\SelectTips{cm}{}

\newtheorem{theorem}{Theorem}[section]
\newtheorem{lemma}[theorem]{Lemma}
\newtheorem{proposition}[theorem]{Proposition}
\newtheorem{corollary}[theorem]{Corollary}
\newtheorem{definition}[theorem]{Definition}

\newcommand{\rotatedarrow}{\mathrel{\reflectbox{\rotatebox[origin=c]{135}{$\Rightarrow$}}}}

\newcommand{\Gpd}{\ensuremath{\mathrm{Gpd}}}
\newcommand{\Cs}{\ensuremath{\mathscr{C}}}
\newcommand{\Xs}{\ensuremath{\mathscr{X}}}
\newcommand{\Cocone}{\ensuremath{\mathrm{Cocone}}}
\newcommand{\Ec}{\ensuremath{\mathcal{E}}}
\newcommand{\Fc}{\ensuremath{\mathcal{F}}}
\newcommand{\Eq}{\ensuremath{\mathrm{Eq}}}
\newcommand{\Lan}{\ensuremath{\mathrm{Lan}}}
\newcommand{\Ext}{\ensuremath{\mathrm{Ext}_\Ec}}
\newcommand{\ExtS}{\ensuremath{\mathrm{Ext}_{\mathrm{Split}(\Ec)}}}
\newcommand{\MExt}{\ensuremath{\mathrm{MExt}_\Ec}}
\newcommand{\CExt}{\ensuremath{\mathrm{CExt}_\Gamma}}
\newcommand{\NExt}{\ensuremath{\mathrm{NExt}_\Gamma}}
\newcommand{\Spl}{\ensuremath{\mathrm{Spl}_\Gamma}}
\newcommand{\Triv}{\ensuremath{\mathrm{TExt}_\Gamma}}
\newcommand{\sTriv}{\ensuremath{\mathrm{TExt}_{\Gamma_{\mathrm{Split}}}}}
\newcommand{\Gal}{\ensuremath{\mathrm{Gal}}}
\newcommand{\ad}[4]{\xymatrix{#1 \ar@/^2ex/[r]^-{#3}  \ar@{}[r]|-\bot &#2 \ar@/^2ex/[l]^-{#4}}}
\newcommand{\catequivalence}[4]{\xymatrix{#1 \ar@/^2ex/[r]^-{#3}  \ar@{}[r]|-\simeq &#2 \ar@/^2ex/[l]^-{#4}}}

\title{A classification theorem for normal extensions}

\author{M.~Duckerts-Antoine}
\address{Centre for Mathematics, University of Coimbra, Departamento de Matem\'atica, Apartado 3008, 3001-501 Coimbra, Portugal}
\email{mathieud@mat.uc.pt}

\author{T.~Everaert}
\address{Institut de Recherche en Math\'ematique et Physique, Universit\'e catholique de Louvain, Chemin du Cyclotron, 2 bte L7.01.01, 1348 Louvain-la-Neuve, Belgium}
\email{tomas.everaert@uclouvain.be}

\thanks{The first author was partially supported by the Universit\'e catholique de Louvain, by the Centro de Matem\'atica da Universidade de Coimbra (CMUC), funded by the European Regional Development Fund through the program COMPETE, and by the Portuguese Government through the FCT - Funda\c{c}\~ao para a Ciˆ\^encia e a Tecnologia under the project PEst-C/MAT/UI0324/2013 and the grant number SFRH/BPD/98155/2013.}

\begin{document}

\begin{abstract}
For a particular class of Galois structures, we prove that the normal extensions are precisely those extensions that are ``locally''  split epic and trivial, and we use this to prove a ``Galois theorem'' for normal extensions. Furthermore, we interpret the normalisation functor as a Kan extension of the trivialisation functor.
\end{abstract}

\maketitle 

\section{Introduction}
For an admissible Galois structure  $\Gamma=(\Cs,\Xs, I,H ,\eta,\epsilon, \Ec,\Fc)$, the Fundamental Theorem \cite{Janelidze:Pure} provides, for every monadic extension $p\colon E\to B$, an equivalence
\[
\Spl(E,p) \simeq \Xs^{\downarrow_{\Fc} \Gal_{\Gamma}(E,p)}
\]
between the category of central extensions (= coverings) of $B$ that are split by $(E,p)$ and the category of discrete fibrations $G\to\Gal_{\Gamma}(E,p)$ of (pre)groupoids in $\Xs$ over the Galois (pre)groupoid
$\Gal_{\Gamma}(E,p)$, with components in the class $\Fc$. When, moreover,  $p\colon E\to B$ is such that it factors through every other monadic extension of $B$ (i.e.\ when it is \emph{weakly universal}), then \emph{every} central extension of $B$ is split by $(E,p)$, and the above equivalence becomes
\[
\CExt(B) \simeq \Xs^{\downarrow_{\Fc} \Gal_{\Gamma}(E,p)}.
\]
Now, as follows from Lemma \ref{localsections} below, this restricts to an equivalence
\[
\CExt(B) \cap \MExt(B)  \simeq \Xs^{\downarrow_{\mathrm{Split}(\Fc)} \Gal_{\Gamma}(E,p)}
\]
between the category of all monadic central extensions of $B$ and that of those discrete fibrations $G\to\Gal_{\Gamma}(E,p)$ whose components are not only in $\Fc$, but are also split epimorphisms. 

In particular, if $\Gamma$ is such that every monadic central extension is normal, the latter equivalence becomes
\begin{equation}\label{mainequivalence}
\NExt(B)  \simeq \Xs^{\downarrow_{\mathrm{Split}(\Fc)} \Gal_{\Gamma}(E,p)}.
\end{equation}
Examples of admissible Galois structures $\Gamma$ for which every monadic central extension is normal, are given by any Birkhoff subcategory (= a reflective subcategory closed under subobjects and regular quotients) $\Xs$ of an exact Mal'tsev category $\Cs$, for $\Ec$ and $\Fc$ the classes of regular epimorphisms in $\Cs$ and $\Xs$, respectively (see \cite{Janelidze-Kelly}). Hence, in this case, the equivalence (\ref{mainequivalence}) holds for every weakly universal monadic extension $p\colon E\to B$. 

The observation we wish to make here is that there is a much larger class of Galois structures $\Gamma$ for which the equivalence (\ref{mainequivalence}) holds for every weakly universal monadic extension, and that such a $\Gamma$ need neither be admissible nor satisfy the condition that every monadic central extension is normal, in general. Among such Galois structures, there is every $\Gamma=(\Cs,\Xs, I,H ,\eta,\epsilon, \Ec,\Fc)$ such that

\begin{itemize}
\item
$\Cs$ is an additive category, $\Xs$ is an arbitrary full reflective subcategory of $\Cs$, and $\Ec$ and $\Fc$ are the classes of all morphisms in $\Cs$ and $\Xs$, respectively;

\item
more generally, $\Cs$ is a pointed protomodular category, $\Xs$ is a reflective subcategory of $\Cs$ with a protoadditive \cite{EG-honfg} reflector $I$, and $\Ec$ and $\Fc$ are the classes of all morphisms in $\Cs$ and $\Xs$, respectively; 

\item
$\Cs$ is an exact Mal'tsev category, $\Xs$ is a Birkhoff subcategory of $\Cs$, and $\Ec$ and $\Fc$ are the classes of regular epimorphisms in $\Cs$ and $\Xs$, respectively---this is the case mentioned above.
\end{itemize}

In each of these cases, the following two conditions are satisfied, and we will show that under these two assumptions the equivalence (\ref{mainequivalence}) is always valid
\begin{itemize}
\item
the left-adjoint functor $I\colon \Cs\to \Xs$ preserves those pullback-squares
\[\xymatrix{
D \ar[r] \ar[d] & A \ar[d]^-f \\
C \ar[r]_-g & B}
\]
for which $f$ is a split epimorphism and $f$ and $g$ are in $\Ec$. 
\item
the induced Galois structure 
\[
\Gamma_{\mathrm{Split}}=(\Cs,\Xs, I,H ,\eta,\epsilon, \mathrm{Split}(\Ec),\mathrm{Split}(\Fc))
\]
is admissible, where the classes $\mathrm{Split}(\Ec)$ and $\mathrm{Split}(\Fc)$ consist of those morphisms in $\Ec$ and $\Fc$, respectively, that are also split epimorphisms.
\end{itemize}
 
In fact, in each of these cases the equivalence (\ref{mainequivalence}) not only holds for every weakly universal monadic extension, but for any weakly universal \emph{normal} extension $p\colon E\to B$ as well. The existence, for every $B$, of such a $p$ is related to that of a left adjoint to the inclusion functor $\NExt(\Cs)\to \Ext(\Cs)$ of the category of normal extensions into that of extensions, and we conclude the article with a closer look at this left adjoint. In particular, we explain how it can be viewed as a Kan extension of the ``trivialisation functor'', and we give a criterion for its existence based on this idea. 

\section{A characterisation of normal extensions}
Recall that a \emph{Galois structure} \cites{Janelidze:Pure,Janelidze:Recent} $\Gamma=(\Cs,\Xs, I,H ,\eta,\epsilon, \Ec,\Fc)$ consists of an adjunction 
\[
\ad{\Cs^{}}{\Xs^{}}{I}{H}
\]
with unit and counit 
\[
\eta\colon 1_\Cs \Rightarrow HI \quad \textrm{and} \quad \epsilon \colon IH \Rightarrow 1_\Xs,
\]
and two classes $\Ec$ and $\Fc$ of morphisms of $\Cs$ and $\Xs$, respectively. $\Ec$ and $\Fc$ are required to be closed under pullback and composition, and to contain all isomorphisms, and one asks that $I(\Ec)\subseteq \Fc$ and $H(\Fc)\subseteq \Ec$. Throughout, we shall call the morphisms $f\colon A\to B$ in the class $\Ec$ \emph{extensions} (of $B$) and write $(\Cs\downarrow_{\Ec} B)$ and $(\Xs\downarrow_{\Fc} Y)$ for the full subcategories of the comma categories $(\Cs\downarrow B)$ and $(\Xs\downarrow Y)$ determined by $\Ec$ and $\Fc$, respectively (for $B\in \Cs$ and $Y\in \Xs$). 

With respect to $\Gamma$, an extension $f\colon A\to B$ is said to be 
\begin{itemize}
\item
 \emph{trivial} if the naturality square
\[
\xymatrix{
A \ar[r]^-{\eta_A}\ar[d]_-f&HI(A)\ar[d]^-{HI(f)}\\
B\ar[r]_-{\eta_B} &HI(B)
}
\]
is a pullback;
\item
\emph{monadic} if the change-of-base functor $f^*\colon(\Cs \downarrow_\Ec B) \rightarrow (\Cs\downarrow_\Ec A)$ is monadic;
\item
\emph{central} (or \emph{a covering}) if it is ``locally'' trivial: there exists a monadic extension $p\colon E\to B$ such that $p^*(f)$ is a trivial extension; in this case one says that $f$ is \emph{split by $p$};
\item
\emph{normal} if it is a monadic extension and if it is split by itself, i.e.\ $f^*(f)$ is trivial.
\end{itemize}

We denote by $\Triv(\Cs)$, $\MExt(\Cs)$, $\CExt(\Cs)$ and $\NExt(\Cs)$ the full subcategories of $\Ext(\Cs)$ given by the trivial-, the monadic-, the central-, and the normal extensions, respectively, and by $\Triv(B)$, etc., the corresponding full subcategories of the comma category $(\Cs\downarrow B)$ (for $B\in \Cs$). For a given monadic extension $p\colon E\to B$, the full subcategory of $(\Cs \downarrow_\Ec B)$ whose objects are split by $p$ will be denoted $\Spl(E,p)$.

By definition, an extension is central when it is ``locally'' trivial. As it turns out, it is \emph{monadic} precisely when it is ``locally'' \emph{split epic} (but this would not make sense as a definition, of course!).

\begin{lemma}\label{localsections}
An extension $f\colon A\to B$ is monadic if and only if there exists a monadic extension $p\colon E\to B$ such that $p^*(f)$ is a split epimorphism.
\end{lemma}
\begin{proof}
For the ``only if'' part, it suffices to take $p=f$. The other implication follows easily by applying Beck's monadicity theorem (see, for instance, \cite{MacDonald-Sobral-CategoricalFoundations}*{Theorem 2.4}).
\end{proof}

In the present article, we are particularly interested in those Galois structures  $\Gamma=(\Cs,\Xs, I,H ,\eta,\epsilon, \Ec,\Fc)$ for which the left-adjoint functor $I\colon \Cs\to \Xs$ preserves all pullback-squares
\begin{equation}\label{square}
\vcenter{\xymatrix{
D \ar[r] \ar[d] & A \ar[d]^-f \\
C \ar[r]_-g & B}}
\end{equation}
for which $f$ is a split epimorphism and $f$ and $g$ are in $\Ec$.  For such a $\Gamma$, we have that a normal extension is the same as a morphism which is ``locally'' a split epic trivial extension. To see this, first of all notice that a simple pullback-cancellation/composition argument yields
\begin{lemma}\label{trivialstable1}
If $I\colon \Cs\to \Xs$ preserves those pullbacks \eqref{square} for which $f$ is a split epimorphism and $f$ and $g$ are in $\Ec$, then trivial extensions which are also split epimorphisms are stable under pullback. 
\end{lemma}

Next, recall (for instance, from \cite{Janelidze-Sobral-Tholen}*{Proposition 1.6}) the following
\begin{lemma}\label{BarrKock}
Consider a commutative diagram
\[
\xymatrix{
A \ar[d]_a \ar[r] & B \ar[r] \ar[d]_b & C \ar[d]^c\\
A' \ar[r]_{f'} & B' \ar[r] & C'}
\]
with $a, b, c\in\Ec$, and assume that $f'^* \colon ( \Cs \downarrow_\Ec B')\to (\Cs \downarrow_\Ec A')$ reflects isomorphisms. The right-hand square is a pullback as soon as both the left-hand square and the outer rectangle are pullbacks. 
\end{lemma}

We are now in a position to prove

\begin{proposition}\label{characterisationnormal}
Assume that $I\colon \Cs\to \Xs$ preserves those pullbacks \eqref{square} for which $f$ is a split epimorphism and $f$ and $g$ are in $\Ec$. Then, for any  $f\colon A\to B$ in $\Ec$, the following are equivalent
\begin{enumerate}
\item\label{a1}
$f$ is a normal extension;
\item \label{a2}
there exists a monadic extension $p\colon E\to B$ such that $p^*(f)$ is both a trivial extension (i.e.\ $f\in \Spl(E,p)$) and a split epimorphism.
\end{enumerate}
\end{proposition}
\begin{proof}
To see that \ref{a1} implies \ref{a2}, it suffices to take $p=f$.

Conversely, let $f$ and $p$ be as in \ref{a2}, and consider the commutative diagram
\[
\xymatrix{
E\times_BA\times_BA \ar[r] \ar[d]_{\bar p_1} &  A\times_BA \ar[d]^{p_1} \ar[r]^-{\eta_{A\times_BA}} & HI(A\times_BA) \ar[d]^{HI(p_1)} \\
E\times_BA \ar[r]^{\bar p} \ar[d]_{\bar f} & A \ar[d]^f \ar[r]_{\eta_A} & HI(A)\\
E \ar[r]_p & B &}
\]
where the two squares on the left are pullbacks, and where $p_1$ and $\bar p_1$ are kernel pair projections of $f$ and $\bar f=p^*(f)$, respectively. We must prove that the remaining square is a pullback as well. By Lemma \ref{BarrKock}, it will suffice if we show the upper rectangle to be a pullback: indeed, since $p$ is a monadic extension, $p^*\colon (\Cs \downarrow_\Ec B)  \to (\Cs \downarrow_\Ec E)$ reflects isomorphisms, and this implies that the same must be true for ${\bar p}^*\colon (\Cs \downarrow_\Ec A)  \to (\Cs \downarrow_\Ec (E\times_BA))$. 

To see that the upper rectangle is indeed a pullback, note that it coincides with the outer rectangle of the commutative diagram  
\[
\xymatrix{
E\times_BA\times_BA \ar[r] \ar[d]_{\bar p_1} & HI(E\times_BA\times_BA) \ar[r] \ar[d]^{HI(\bar p_1)} &  HI(A\times_BA) \ar[d]^{HI(p_1)} \\
E\times_BA \ar[r]_-{\eta_{E\times_BA}} & HI(E\times_BA) \ar[r]_{HI(\bar p)} & HI(A).}
\]
Here, the right-hand square is the image under $HI$ of the left-hand upper square in the previous diagram, which is a pullback preserved by $I$, hence by $HI$; the left-hand square is induced by the unit $\eta$ and is a pullback, since $\bar p_1$ is a trivial extension by Lemma \ref{trivialstable1}.

Note, finally, that $f$ is a monadic extension by Lemma~\ref{localsections}.
\end{proof}

If we write $\mathrm{SSpl}_{\Gamma}(E,p)$ for the full subcategory of $\Spl(E,p)$ consisting of those $(A,f)\in\Spl(E,p)$ for which $p^*(f)$ is a split epimorphism, then Proposition \ref{characterisationnormal} may be expressed as an equality
\[
\NExt(B) = \bigcup_p \mathrm{SSpl}(E,p)
\]
where $p$ runs through all monadic extensions $p\colon E\to B$ of $B$. If there is a single $p\colon E\to B$ with $\mathrm{SSpl}(E',p')\subseteq \mathrm{SSpl}(E,p)$ for every monadic extension $p'\colon E'\to B$ of $B$, this equality moreover simplifies to   
\[
\NExt(B) = \mathrm{SSpl}(E,p).
\]
Such a $p$ often exists (assuming we are in the situation of Proposition \ref{characterisationnormal}): since split epic trivial extensions are stable under pullback (by Lemma \ref{trivialstable1}), examples are given by any $p\colon E\to B$ which factors through every normal extension of $B$. For instance, $p$ could be a weakly universal monadic extension of $B$ (= a weakly initial object of $\MExt(B)$) or a weakly universal normal extension (= a weakly initial object of $\NExt(B)$).

\section{The classification theorem}

Recall that an \emph{internal groupoid} $G$ in a category $\Cs$ is a diagram of the form
\begin{equation}\label{groupoid}
\xymatrix{
G_2 \ar@<1ex>[rr]^-{p_1}  \ar[rr]|-m  \ar@<-1ex>[rr]_-{p_2}  && G_1 \ar@(ul,ur)[]^-\sigma \ar@<1ex>[rr]^-{d}  \ar@<-1ex>[rr]_-{c} && G_0 \ar[ll]|-e 
}
\end{equation}
with 
\begin{equation}\label{pbgp}
\vcenter{
\xymatrix{
G_2 \ar[r]^-{p_2}  \ar[d]_-{p_1} &G_1\ar[d]^-{d} \\
G_1\ar[r]_-{c} &G_0
}}
\end{equation}
a pullback and such that $de=1_{}=ce$, $dm=dp_1$, $cm=cp_2$, $m(1_{},ec)=1_{}=m(ed,1_{})$,  $m(1_{}\times_{}m)=m( m\times_{}1_{})$, $d\sigma=c$, $c\sigma=d$, $m(1_{},\sigma)=ed$, and $m(\sigma, 1_{})=ec$.           

An \emph{internal functor} $f \colon G' \to G$ between groupoids $G'$ and $G$ in $\Cs$ is a triple $(f_0\colon G_0' \to G_0,f_1\colon G_1'\to G_1, f_2\colon G_2'\to G_2)$ of morphisms such that the evident squares in the diagram       \[
\xymatrix{
G_2'  \ar[d]_{f_2} \ar@<1ex>[rr]^-{p_1'}  \ar[rr]|-m  \ar@<-1ex>[rr]_-{p_2'}  &&G_1'  \ar@<1ex>[rr]^-{d'}  \ar@<-1ex>[rr]_-{c'} \ar[d]^-{f_1} && G_0' \ar[ll]|-{e'} \ar[d]^-{f_0} \\
G_2  \ar@<1ex>[rr]^-{p_1}  \ar[rr]|-{m}  \ar@<-1ex>[rr]_-{p_2}  &&G_1  \ar@<1ex>[rr]^-{d}  \ar@<-1ex>[rr]_-{c} && G_0 \ar[ll]|-{e} 
}
\]
 commute (from which it follows immediately that also $\sigma f_1'=f_1\sigma$). $f$ is a \emph{discrete fibration} when those commutative squares are moreover pullbacks. (Note that it suffices for this that the square $f_0c'=cf_1$ is a pullback.)

The category of groupoids and functors in $\Cs$ will be denoted by  $\Gpd(\Cs)$ and, for a fixed groupoid $G$, the full subcategory of the comma category ${(\Gpd(\Cs)\downarrow G)}$ given by the discrete fibrations $G' \to G$, by  $\Cs^{G}$. When $\Ec$ is a class of morphisms in $\Cs$, then $\Cs^{\downarrow_{\Ec} G}$ will denote the full subcategory of $\Cs^{G}$ of those  $(f_0,f_1,f_2)\colon G'\to G$ for which $f_0$, $f_1$ and $f_2$ are in $\Ec$. 

Any internal equivalence relation is a groupoid, which means in particular that every morphism $p\colon E\to B$ determines, via its kernel pair $(\pi_1^p,\pi_2^p)$, an internal groupoid $\Eq(p)$, as in the diagram
\[
\xymatrix{
\Eq(p) \times_E \Eq(p) \ar@<1ex>[rr]^-{p_1^p}  \ar@<-1ex>[rr]_-{p_2^p}  \ar[rr]|-{\tau} &&\Eq(p) \ar@<1ex>[rr]^-{\pi_1^p} \ar@<-1ex>[rr]_-{\pi_2^p} \ar@(ul,ur)[]^-{\sigma} && E\ar@{>}[ll]|-{\delta}.
}
\]
Notice, for any discrete fibration of groupoids  $f\colon G'\to G$, that $G'$ is an equivalence relation as soon as $G$ is.

Let us, from now on, consider a Galois structure $\Gamma=(\Cs,\Xs, I,H ,\eta,\epsilon, \Ec,\Fc)$ such that 
\begin{itemize}
\item
\emph{the left-adjoint functor $I\colon \Cs\to \Xs$ preserves those pullback-squares \eqref{square} for which $f$ is a split epimorphism and $f$ and $g$ are in $\Ec$.}
\end{itemize}
If $p\colon E\to B$ is an extension, then so are its kernel pair projections $\pi_1^p$ and $p_2^p$, hence $I$ preserves the pullback (\ref{pbgp}) for $G=\Eq(p)$.  Consequently, $I(\Eq(p))$ is again a groupoid, in $\Xs$, called the \emph{Galois groupoid} of $p$. We denote it $\Gal_{\Gamma}(E,p)$.

What we wish to prove now is that there is, under the additional condition~2. below, for every monadic extension $p\colon E\to B$ an equivalence
\begin{equation}\label{unnormalclassification}
\mathrm{SSpl}_{\Gamma}(E,p) \simeq \Xs^{\downarrow_{\mathrm{Split}(\Fc)}\Gal_{\Gamma}(E,p)}
\end{equation}
between the category of extensions $(A,f)$ of $B$ for which $p^*(f)$ is a split epic trivial extension, and the category of discrete fibrations 
\[
(f_0,f_1,f_2)\colon G' \to\Gal_{\Gamma}(E,p)
\]
in $\Xs$ whose components $f_0$, $f_1$ and $f_2$ are split epimorphisms and are in $\Fc$. When $p$ is such that $\mathrm{SSpl}(E,p)=\NExt(B)$  (for instance, if $p$ is a weakly universal monadic extension, or a weakly universal normal extension---see the end of the previous section) we then obtain
\[
\NExt(B) \simeq \Xs^{\downarrow_{\mathrm{Split}(\Fc)}\Gal_{\Gamma}(E,p)}
\]

Fix an extension $p\colon E\to B$.  By sending any extension $(A,f)$ of $B$ to the discrete fibration induced by the right-hand pullback square in, and displayed as the left-hand side of, the diagram
\[\vcenter{
\xymatrix{
\Eq(\bar{p}) \times_P \Eq(\bar{p}) \ar[d]  \ar@<1ex>[rr]^-{p_1^{\bar{p}}}  \ar@<-1ex>[rr]_-{p_2^{\bar{p}}}  \ar[rr]|-{} &&\Eq(\bar{p}) \ar[d] \ar@<1ex>[rr]^-{\pi_1^{\bar{p}}} \ar@<-1ex>[rr]_-{\pi_2^{\bar{p}}} && P\ar@{>}[ll]|-{}  \ar[rr]^-{\bar{p}=f^*(p)} \ar[d]^{p^*(f)} && A \ar[d]^f\\
\Eq(p) \times_E \Eq(p) \ar@<1ex>[rr]^-{p_1^p}  \ar@<-1ex>[rr]_-{p_2^p}  \ar[rr]|-{} &&\Eq(p) \ar@<1ex>[rr]^-{\pi_1^p} \ar@<-1ex>[rr]_-{\pi_2^p}  && E\ar@{>}[ll]|-{}  \ar[rr]_p && B,}}
\]
we obtain a functor $K^p\colon (\Cs\downarrow_{\Ec} B)\to \Cs^{\downarrow_{\Ec} \Eq(p)}$. It turns out  (see, for instance, \cite{Janelidze-Tholen,Janelidze-Tholen2}) that $\Cs^{\downarrow_{\Ec} \Eq(p)}$ is equivalent to the category $(\Cs\downarrow_{\Ec} E)^{T^p}$ of (Eilenberg-Moore) algebras for the monad $T^p=p^*\Sigma_p$, and that $K^p\colon (\Cs\downarrow_{\Ec} B)\to \Cs^{\downarrow_{\Ec} \Eq(p)}$ corresponds, via this equivalence, to the comparison functor $K^{T^p}\colon (\Cs\downarrow_{\Ec} B)\to (\Cs\downarrow_{\Ec} E)^{T^p}$,
whence

\begin{lemma}\label{monadicdiscreet}\cites{Janelidze-Tholen,Janelidze-Tholen2}
An extension $p\colon E\to B$ is  monadic if and only if the functor $K^p\colon (\Cs\downarrow_{\Ec} B)\to \Cs^{\downarrow_{\Ec} \Eq(p)}$ is an equivalence of categories. 
\end{lemma}
$K^p$ restricts to an equivalence
\begin{equation}\label{firsthalf}
\mathrm{SSpl}_{\Gamma}(E,p) \simeq \Cs^{\downarrow_{\sTriv(\Cs)}\Eq(p)}
\end{equation}
between the category of extensions $(A,f)$ of $B$ for which $p^*(f)$ is a split epic trivial extension, and the category of discrete fibrations $(f_0,f_1,f_2)\colon G' \to\Eq(p)$ in $\Cs$ whose components $f_0$, $f_1$ and $f_2$ are split epic trivial extensions (= $\Gamma_{\mathrm{Split}}$-trivial extensions---see the introduction, or below, for the notation $\Gamma_{\mathrm{Split}}$), and we are already halfway to proving (\ref{unnormalclassification}).

In order to find an equivalence
\begin{equation}\label{secondhalf}
\Cs^{\downarrow_{\sTriv(\Cs)}\Eq(p)}\simeq \Xs^{\downarrow_{\mathrm{Split}(\Fc)}\Gal_{\Gamma}(E,p)},
\end{equation}
first of all notice that, since $I(\Ec)\subseteq \Fc$, the reflector $I\colon \Cs\to \Xs$ extends, for any $B\in\Cs$, to a functor $I^B\colon (\Cs\downarrow_{\Ec} B)\to (\Xs\downarrow_{\Fc} I(B))$ in an obvious way. Because $H(\Fc)\subseteq \Ec$ and since $\Ec$ is stable under pullback, $I^B$ has a right adjoint $H^B\colon (\Xs\downarrow_{\Fc} I(B))\to (\Cs\downarrow_{\Ec} B)$, which sends an $(X,\varphi)\in (\Xs\downarrow_{\Fc} I(B))$ to the extension $(A,f)\in (\Cs\downarrow_{\Ec} B)$ defined via the pullback
\[
\xymatrix{
A \ar[r] \ar[d]_{f} & H(X) \ar[d]^{H(\varphi)}\\
B \ar[r]_-{\eta_B} & HI(B).}
\]  
This gives us, for every $B$ in $\Cs$, an adjunction
\begin{equation}\label{inducedadjunction}
\ad{(\Cs\downarrow_{\Ec}B)}{(\Xs\downarrow_{\Fc}I(B))}{I^B}{H^B}
\end{equation}
which restricts to an equivalence
\[
\Triv(B)\simeq (\Xs\downarrow_{\Fc}I(B))
\]
whenever $H^B$ is fully faithful---a situation which is of interest:

\begin{definition}\cite{Janelidze:Recent}
A Galois structure $\Gamma$ is called \emph{admissible} when each functor $H^B\colon {(\Xs\downarrow_{\Fc} I(B))}  \rightarrow {(\Cs\downarrow_{\Ec} B)}$ is fully faithful. 
\end{definition}

Now since $I\colon \Cs\to\Xs$ preserves those pullback-squares (\ref{square}) for which $f$ is a split epimorphism and $f$ and $g$ are in $\Ec$, the adjunction (\ref{inducedadjunction}) induces an adjunction
\[
\ad{\Cs^{\downarrow_{\Ec}G}}{\Xs^{\downarrow_{\Fc}I(G)}}{}{}
\]
for every groupoid $G$ (as in (\ref{groupoid})) in $\Cs$ with $d$ and $c$ (hence, also $p_1$, $m$ and $p_2$) in $\Ec$, and this, in its turn, would restrict to an equivalence
\[
\Cs^{\downarrow_{\Triv(\Cs)}G}\simeq \Xs^{\downarrow_{\Fc}I(G)}.
\]
if $\Gamma$ were admissible. However, instead of requiring this for $\Gamma$, we only ask that
\begin{itemize}
\item
\emph{the induced Galois structure 
\[
\Gamma_{\mathrm{Split}}=(\Cs,\Xs, I,H ,\eta,\epsilon, \mathrm{Split}(\Ec),\mathrm{Split}(\Fc))
\]
is admissible, where the classes $\mathrm{Split}(\Ec)$ and $\mathrm{Split}(\Fc)$ consist of those morphisms in $\Ec$ and $\Fc$, respectively, that are also split epimorphisms.}
\end{itemize}
In this case, we instead obtain an equivalence 
\[
\Cs^{\downarrow_{\sTriv(\Cs)}G}\simeq \Xs^{\downarrow_{\mathrm{Split}(\Fc)}I(G)}
\]
for every groupoid $G$ in $\Cs$ with $d$ and $c$ in $\Ec$. In particular, if $G=\Eq(p)$ for some monadic extension $p\colon E\to B$, we find the sought-after (\ref{secondhalf}).

Combining (\ref{firsthalf}) and (\ref{secondhalf}), we obtain:

\begin{theorem}\label{classificationtheorem}
Assume that $\Gamma=(\Cs,\Xs, I,H ,\eta,\epsilon, \Ec,\Fc)$ is a Galois structure such that
\begin{enumerate}
\item \label{pres}
the left adjoint $I\colon \Cs\to\Xs$ preserves those pullbacks \eqref{square} for which  $f$ is a split epimorphism and $f$ and $g$ are in $\Ec$.
\item \label{restad}
the induced Galois structure 
\[
\Gamma_{\mathrm{Split}}=(\Cs,\Xs, I,H ,\eta,\epsilon, \mathrm{Split}(\Ec),\mathrm{Split}(\Fc))
\]
is admissible.
\end{enumerate}
Then, for any monadic extension $p\colon E\to B$, there is an equivalence of categories
\[
\mathrm{SSpl}_{\Gamma}(E,p) \simeq \Xs^{\downarrow_{\mathrm{Split}(\Fc)}\Gal_{\Gamma}(E,p)}.
\] 
Hence, if $p$ is such that $\mathrm{SSpl}(E,p)=\NExt(B)$ (for instance, if it is a weakly universal monadic extension, or a weakly universal normal extension), there is a category equivalence
\[
\NExt(B) \simeq \Xs^{\downarrow_{\mathrm{Split}(\Fc)}\Gal_{\Gamma}(E,p)}.
\] 
\end{theorem}

Weakly universal monadic extensions often exist: for instance, if $\Cs$ is a Barr exact category \cite{Barr} with enough (regular) projectives, and $\Ec$ is either the class of regular epimorphisms or the class of \emph{all} morphisms (in either case the monadic extensions are precisely the regular epimorphisms), then clearly \emph{every} $B$ admits a weakly universal monadic extension. It turns out that the existence of a weakly universal monadic extension $p\colon E\to B$ at once implies that of a weakly universal \emph{normal} extension of $B$, if we are in the situation of Theorem \ref{classificationtheorem}. Indeed, we have

\begin{proposition}\label{prnormalisation}
If $\Gamma=(\Cs,\Xs, I,H ,\eta,\epsilon, \Ec,\Fc)$ satisfies the assumptions of Theorem \ref{classificationtheorem}, and if there exists, for a given object $B$ of $\Cs$,  a weakly universal monadic extension $p\colon E\to B$, then the inclusion functor $\NExt(B)\to \MExt(B)$ admits a left adjoint. 

If there is such a $p$ for every $B$, and if monadic extensions are stable under pullback, then also the inclusion functor $\NExt(\Cs)\to \MExt(\Cs)$ has a left adjoint. 
\end{proposition}
\begin{proof}
Since $\Gamma_{\mathrm{Split}}$ is admissible, for every $B$ in $\Cs$ the adjunction (\ref{inducedadjunction}) induces a \emph{reflection}
\[
\ad{(\Cs\downarrow_{\mathrm{Split}(\Ec)}B)}{\sTriv(B).}{}{}
\]
Because $I\colon \Cs\to\Xs$ preserves those pullback-squares (\ref{square}) for which $f$ is a split epimorphism and $f$ and $g$ are in $\Ec$, these reflections, in their turn, induce a reflection
\[
\ad{\Cs^{\downarrow_{\mathrm{Split}(\Ec)}G}}{  \Cs^{\downarrow_{ \sTriv(\Cs) }G}}{}{}
\]
for every groupoid $G$ (as in (\ref{groupoid})) in $\Cs$, with $d$ and $c$ (hence, also $p_1$, $m$ and $p_2$) in $\Ec$. In particular, if $G=\Eq(p)$ for some  extension $p\colon E\to B$, we have a reflection
\begin{equation}\label{reflectivenessadjunction}
\ad{\Cs^{\downarrow_{\mathrm{Split}(\Ec)}\Eq(p)}}{  \Cs^{\downarrow_{ \sTriv(\Cs) }\Eq(p)  }  .}{}{}
\end{equation}
To prove our first claim, it suffices now to observe that the inclusion functor in (\ref{reflectivenessadjunction}) coincides, up to equivalence, with the inclusion functor $\NExt(B)\to \MExt(B)$ whenever $p$ is a weakly universal monadic extension: indeed, in this case, an extension $f\colon A\to B$ is monadic if and only if $p^*(f)$ is a split epimorphism (by Lemma \ref{localsections}), and $f$ is a normal extension if and only if $p^*(f)$ is, moreover, a trivial extension (by Proposition \ref{characterisationnormal} and Lemma \ref{trivialstable1}). Thus, the equivalence $K^p\colon (\Cs\downarrow_{\Ec} B)\to \Cs^{\downarrow_{\Ec} \Eq(p)}$ restricts to equivalences $\MExt(\Cs)\to \Cs^{\downarrow_{\mathrm{Split}(\Ec)}\Eq(p)}$ and $\NExt(B)\to    \Cs^{\downarrow_{ \sTriv(\Cs) }\Eq(p)}$, and the inclusion functor in (\ref{reflectivenessadjunction}) to the inclusion functor $\NExt(B)\to \MExt(B)$. 

The second claim follows from the first by Proposition 5.8 in \cite{Im-Kelly}, since the stability under pullback of monadic extensions implies that of normal extensions, by Lemma \ref{trivialstable1}. 
\end{proof}

Notice that, whenever $p\colon E\to B$ is a weakly universal monadic extension, its \emph{normalisation} (=its reflection in $\NExt(B)$) must be weakly universal too. 

The reflector into $\NExt(\Cs)$ is our object of study in the next section. Here, we want to add that if we drop the assumption that a weakly universal monadic extension $p\colon E\to B$ exists for every $B$, but instead require every extension to be monadic, we still have a reflector $\MExt(\Cs)=\Ext(\Cs)\to \NExt(\Cs)$. Before proving this, we note that  Lemma \ref{trivialstable1}  remains valid in this situation (see, for instance, \cite{JK:Reflectiveness}*{Proposition 2.1}).

\begin{lemma}\label{trivialstable2}
If $\Gamma_{\mathrm{Split}}$ is admissible, then trivial extensions which are also split epimorphisms are stable under pullback. Consequently, if a pullback of a normal extension is monadic, it is a normal extension as well.
\end{lemma}

\begin{proposition}\label{goodcriterionnormalisation}
Assume that $\Gamma=(\Cs,\Xs, I,H ,\eta,\epsilon, \Ec,\Fc)$ satisfies the assumptions of Theorem \ref{classificationtheorem} and that  every extension is monadic. The inclusion functor $\NExt(\Cs)\to \Ext(\Cs)$ admits a left adjoint. 
\end{proposition}

\begin{proof}
Let $f\colon A\to B$ be an extension.  As in the proof of Proposition \ref{prnormalisation}, we have a reflection 
\[
\ad{\Cs^{\downarrow_{\mathrm{Split}(\Ec)}\Eq(f)}}{\Cs^{\downarrow_{ \sTriv(\Cs) }\Eq(f)}}{}{}
\]
since $\Gamma_{\mathrm{Split}}$ is admissible, and because $I\colon \Cs\to\Xs$ preserves those pullback-squares (\ref{square}) for which $f$ is a split epimorphism and $f$ and $g$ are in $\Ec$. Furthermore, the equivalence $K^f\colon (\Cs\downarrow_{\Ec} B)\to \Cs^{\downarrow_{\Ec} \Eq(f)}$ induces an equivalence
\[
(f \downarrow (\Cs\downarrow_{\Ec} B)) \simeq  (K^f(f) \downarrow (\Cs^{\downarrow_{\Ec} \Eq(f)})) =  (K^f(f) \downarrow (\Cs^{\downarrow_{\mathrm{Split}( \Ec)} \Eq(f)}))
\]
which, by Proposition \ref{characterisationnormal} and Lemma \ref{trivialstable1}, restricts to an equivalence 
\[
(f\downarrow \NExt(B)) \simeq (K^f(f) \downarrow \Cs^{\downarrow_{\sTriv(\Cs) }\Eq(f)}).
\]
Since $K^f(f)$, as an object of $\Cs^{\downarrow_{\mathrm{Split}(\Ec)}\Eq(f)}$, has a reflection in $\Cs^{\downarrow_{ \sTriv(\Cs) }\Eq(f)}$, both categories $(K^f(f) \downarrow \Cs^{\downarrow_{\sTriv(\Cs) }\Eq(f)})$ and $(f\downarrow \NExt(B))$ have an initial object. Consequently, $f$ has a reflection in $\NExt(B)$. Finally, since normal extensions are stable under pullback by Lemma \ref{trivialstable2}, we can apply Proposition 5.8 in \cite{Im-Kelly} and conclude that  the inclusion $\NExt(\Cs)\to \Ext(\Cs)$ has a left adjoint.
\end{proof}

In concluding this section, let us return to what we wrote in the introduction. The examples of Galois structures given there do indeed satisfy conditions \ref{pres} and \ref{restad} of Theorem \ref{classificationtheorem}: the former is well known to hold in the case of an additive $I\colon\Cs\to \Xs$ (between additive categories $\Cs$ and $\Xs$), and remains valid for a protoadditive $I\colon\Cs\to \Xs$ (between pointed protomodular $\Cs$ and $\Xs$), by Proposition 2.2 in \cite{Everaert-Gran-TT}. By Proposition 3 and Example 1 in \cite{EverMaltsev}, it also holds if $I\colon\Cs\to \Xs$ is the reflector into a Birkhoff subcategory $\Xs$ of an exact Mal'tsev category $\Cs$. Moreover, condition \ref{pres} implies condition \ref{restad}, in each of these cases:
\begin{itemize}
\item
when the adjunction $I \dashv H$ is a reflection, the admissibility of $\Gamma_{\mathrm{Split}}$ can equivalently be described as the preservation by $I\colon \Cs\to \Xs$ of every pullback  (\ref{square}) for which $f$ is in $\mathrm{Split}(\Ec)$ and $g=\eta_C\colon C\to HI(C)$ is a reflection unit (see Proposition 2.1 in \cite{JK:Reflectiveness}). 
\item
when, moreover, $\eta_B\colon B\to HI(B)$ is an extension for every $B$, we thus have that the first condition  of Theorem \ref{classificationtheorem} implies the second.
\end{itemize}

\section{The normalisation functor as a Kan extension}
Let $\Gamma=(\Cs,\Xs, I,H ,\eta,\epsilon, \Ec,\Fc)$ be a Galois structure such that the induced Galois structure $\Gamma_{\mathrm{Split}}=(\Cs,\Xs, I,H ,\eta,\epsilon, \mathrm{Split}(\Ec),\mathrm{Split}(\Fc))$ is admissible. For every $B\in \Cs$, $I\dashv H$ induces an adjunction 
\[
\ad{(\Cs\downarrow_{\mathrm{Split}(\Ec)}B)}{(\Xs\downarrow_{\mathrm{Split}(\Fc)}I(B))}{}{}
\]
which, by admissibility, decomposes into a reflection followed by an equivalence
\[
\xymatrix{
(\Cs\downarrow_{\mathrm{Split}(\Ec)}B) \ar@/^2ex/[r]^-{}  \ar@{}[r]|-\bot & \sTriv(B) \ar@/^2ex/[l]^-{}  \ar@/^2ex/[r]^-{}  \ar@{}[r]|-\simeq & (\Xs\downarrow_{\mathrm{Split}(\Fc)}I(B)) \ar@/^2ex/[l]^-{}.
}
\]
In particular, we have that the inclusion functor $\sTriv(B) \to \ExtS(B)$ admits a left adjoint, for every $B\in \Cs$. Since, by admissibility of $\Gamma_{\mathrm{Split}}$, split epic trivial extensions are stable under pullback (Lemma \ref{trivialstable2}), we can apply Proposition 5.8 in \cite{Im-Kelly} and conclude that also the inclusion functor $\sTriv(\Cs) \to \ExtS(\Cs)$, which we denote by $\tilde H_1$, has a left adjoint.  We call it $T_1$, and we write $\tilde \eta^1$ for the unit of the adjunction $T_1 \dashv \tilde H_1$. When also the inclusion functor $H_1 \colon \NExt(\Cs) \to \Ext(\Cs)$ has a left adjoint $I_1$ (as, for instance, in Propositions \ref{prnormalisation} and \ref{goodcriterionnormalisation}),  we obtain a square of functors
\begin{equation}\label{squareoffunctors}
\vcenter{\xymatrix{
 \ExtS(\Cs) \ar[r]^-K \ar[d]_-{T_1} & \Ext(\Cs) \ar[d]^-{I_1}\\
 \sTriv(\Cs)\ar[r]_-{\tilde K}&\NExt(\Cs)
}}
\end{equation}
in which $K$ and $\tilde K$ are the inclusion functors. This square commutes, up to natural isomorphism. Indeed, first of all, we have, for any split epic extension $p\colon E\to B$, that its reflection $T_1(f)$ in $\sTriv(B)$ is also its reflection in $\NExt(B)$, since 
\[
(p\downarrow \NExt(B)) = (p \downarrow\sTriv(B))
\]
by the right-cancellation property of split epimorphisms  and by
\begin{lemma}\label{splitnormal=trivial}
If $\Gamma_{\mathrm{Split}}$ is admissible, then a split epic extension is normal if and only if it is trivial.
\end{lemma}
\begin{proof} 
Every split epic extension is monadic (see \cite{Janelidze-Tholen2}) and, by Lemma \ref{trivialstable2}, split epic trivial extensions are stable under pullback, which implies they are normal.

For the converse, it suffices to consider, for any split epic normal extension $f\colon A\to B$, with section $s\colon B \to A$, the diagram
\[
\xymatrix{
A \ar[r]^-{\langle sf , 1\rangle} \ar@<0.5ex>[d]^-f & \Eq(f) \ar[r]^-{\pi_2^f} \ar@<0.5ex>[d]^-{\pi_1^f} & A \ar@<0.5ex>[d]^-f \\
B\ar[r]_-s \ar@<0.5ex>@{.>}[u] & A \ar[r]_-f \ar@<0.5ex>@{.>}[u] & B \ar@<0.5ex>@{.>}[u] 
}
\]
in which each square is a pullback, and to use, again, Lemma \ref{trivialstable2}.
\end{proof}

In order to conclude from this that $T_1(f)$ is also the reflection in $\NExt(\Cs)$ of $f$, for $f$ a split epic trivial extension, we would like to apply, once more, Proposition 5.8 in \cite{Im-Kelly}. While we have no reason to assume that arbitrary normal extensions  are stable under pullback,  the pullback stability given in Lemma \ref{trivialstable2} is easily seen to suffice, here. Whence

\begin{lemma}\label{usefulllemma} Assume that $\Gamma_{\mathrm{Split}}$ is admissible, and let  $p \colon A \to B$ be an object of $\mathrm{Ext}_{\mathrm{Split}(\Ec)}(\Cs)$.  The reflection $T_1(p)$ of $p$ in the category $\sTriv(\Cs)$ is also its reflection in $\NExt(\Cs)$ (irrespective of the existence of $I_1$).
\end{lemma}

In particular, the square (\ref{squareoffunctors}) indeed commutes, up to natural isomorphism, whenever $I_1$ exists.

What we want to prove now is that under the additional condition that every extension is a regular epimorphism, the normalisation functor $I_1$, when it exists,  coincides with the pointwise left Kan extension of $\tilde K \circ T_1$ along $K$. 

 We shall first show that  the functor $K$ is dense, i.e.\ the functor
\[
\xymatrix{
(K\downarrow f) \ar[r]^-{P^f} &\ExtS(\Cs) \ar[r]^-K& \Ext(\Cs)
}
\]
where $P^f$ is the obvious forgetful functor admits $f$ as colimit. For this, let us consider the full subcategory  $J^f$ of $(K\downarrow f)$ determined by
 \[
 \xymatrix{
p_1^f\ar@<0.5ex>[rr]^-{\pi_1^q} \ar@<-0.5ex>[rr]_-{\pi_2^q}  \ar[rdd]_-{r}&&  \pi_1^f \ar[ldd]^-{q} \\\\
& f
}
 \]
 where $q=(\pi_2^f,f)$, $\pi_1^q =(\tau,\pi_1^f)$, $\pi_2^q=(p_2^f,\pi_2^f) $ and $r=q\circ \pi_1^q =q \circ \pi_2^q$. One can easily prove that the inclusion functor $L^f \colon J^f \to (K\downarrow f)$ is final, i.e.~for any object $P=(p,(p_1,p_0) \colon p\to f)$ in $(K\downarrow f)$, the category $(P\downarrow L^f)$ is non-empty and connected. Let us prove here that $(P\downarrow L^f)$ is connected. Any three objects 
    \[
\xymatrix{
p_1^f  \ar[rrdd]_-r && p \ar@<0.5ex>[rr]^-{(\overline{p}_1,\overline{p}_0)} \ar@<-0.5ex>[rr]_-{(\overline{p}_1',\overline{p}_0')}  \ar[ll]_-{\langle (\overline{p}_1'',\overline{p}_0'') ,(\overline{p}_1',\overline{p}_0')\rangle } \ar[dd]|-{(p_1,p_0)} && \pi_1^f \ar[lldd]^-{q} \\\\
 &&f
}
\]
 in $(P\downarrow L^f)$ are connected as shown in the commutative diagram
  \[
 \xymatrix{
p \ar[dd]|-{\langle (\overline{p}_1'',\overline{p}_0'') ,(\overline{p}_1',\overline{p}_0')\rangle}   \ar@{=}[rr]& &  p  \ar[dd]|-{(\overline{p}_1',\overline{p}_0')}  && p \ar@{=}[rr] \ar@{=}[ll] \ar@{.>}[dd]|-{\langle (\overline{p}_1',\overline{p}_0'), (\overline{p}_1,\overline{p}_0)\rangle} &&p\ar[dd]|-{(\overline{p}_1,\overline{p}_0)} \\\\
p_1^f \ar@{.>}[rr]^-{\pi_2^q} \ar@/_/[rrdrd]_-r && \pi_1^f  \ar[ddr]_-{q} && p_1^f \ar@{.>}[rr]^-{\pi_2^q} \ar@{.>}[ll]_-{\pi_1^q} \ar@{.>}[ddl]^-r && \pi_1^f .\ar@/^/[llldd]^-q \\\\
&& &f
 }
 \]
The last case to be considered can easily be deduced from this. Now, from the assumption that $f$ is a regular epimorphism (as is every morphism in $\Ec$), we conclude that $q=(\pi_2^f,f)\colon \pi_1^f \to f$ is the coequaliser of its kernel pair
 \[
\xymatrix{
p_1^f \ar@<0.5ex>[r]^-{\pi_1^q} \ar@<-0.5ex>[r]_-{\pi_2^q}& \pi_1^f.
}
\]
This precisely means that the functor 
\[
\xymatrix{
J^f \ar[r]^-{L^f} & (K\downarrow f) \ar[r]^-{P^f} & \ExtS(\Cs) \ar[r]^-K& \Ext(\Cs)
}
\] 
has $f$ as colimit.

\begin{theorem}\label{normalisationasKanextension} Assume that $\Gamma=(\Cs,\Xs, I,H ,\eta,\epsilon, \Ec,\Fc)$ is a Galois structure such that
\begin{enumerate}
\item
the induced Galois structure 
\[
\Gamma_{\mathrm{Split}}=(\Cs,\Xs, I,H ,\eta,\epsilon, \mathrm{Split}(\Ec),\mathrm{Split}(\Fc))
\]
 is admissible.
\item 
every morphism in $\Ec$ is a regular epimorphism.
\end{enumerate}

Then the inclusion functor $H_1 \colon \NExt(\Cs) \to \Ext(\Cs)$ has a left adjoint $I_1$ if and only if the pointwise left Kan extension of $\tilde K\circ T_1$ along $K$ exists and, in this case, $I_1 =\Lan_K(\tilde K\circ T_1)$. 
\[
\xymatrix{
\Ext(\Cs) \ar@{.>}[rrd]^-{I_1=\Lan_K(\tilde K\circ T_1)} &&\\
\ExtS(\Cs) \ar@{}[rru]|(0.3){\cong} \ar[rr]_{\tilde K \circ T_1} \ar[u]^-{K} & &\NExt(\Cs).
}
\]
\end{theorem}
\begin{proof}
$H_1$ has a left adjoint if and only if for all $f$ in $\Ext(\Cs)$, $(f \downarrow H_1)$ has an initial object.  We are going to show that one has an isomorphism
\[
(f \downarrow H_1) \cong \Cocone(\tilde K\circ T_1 \circ P^f)
\]
for any $f$ in $\Ext(\Cs)$. This allows us to assert that $H_1$ has a left adjoint if and only if the pointwise left Kan extension $\Lan_K(\tilde K\circ T_1)$ exists.
Let $f$ be in $\Ext(\Cs)$ and $\lambda^f=(\lambda^f_p \colon p \to f)_{p\in \mathrm{Ext}_{\mathrm{Split}(\Ec)}(\Cs)}$ be the cocone defined by the comma square
\[
\xymatrix{
(K\downarrow f) \ar[r] \ar[d]_-{P^f} \ar@{}[rd]|-{\rotatedarrow_{\lambda^f}} & 1 \ar[d]^-f \\
\ExtS(\Cs) \ar[r]_-K &  \Ext(\Cs).
}
\]
The density of $K$ implies that one has an isomorphism 
\[
(f\downarrow \Ext(\Cs)) \to \Cocone (K\circ P^f)  
\]
defined on an objet $G= ((g_1,g_0)\colon f \to g, g)$ by
 \[
\lambda^G=( \lambda^G_p= (g_1,g_0)\circ  \lambda^f_p \colon p  \to g)_{p \in \ExtS(\Cs)}.
\] 
Now, when $g$ is in $\NExt(\Cs)$, one can associate with $\lambda^G$ a cocone
\[
\tilde \lambda^G =(\tilde \lambda^G_p\colon \tilde KT_1(p) \to g )_{p\in \ExtS(\Cs)}
\] 
on $\tilde K\circ T_1 \circ P^f$ where $\tilde \lambda^G_p$ is the (unique) factorisation of $\lambda^G_p \colon p \to g$ through $\tilde \eta^1_p\colon p \to T_1(p)=\tilde K T_1(p)$ (see Lemma \ref{usefulllemma}). The assignment $G \mapsto \tilde \lambda^G$ extends to the desired isomorphism. 

Now let us suppose that the normalisation functor $I_1$ exists. Since $K$ is dense, one has a left Kan extension
\[
\xymatrix{
\Ext(\Cs) \ar@{.>}[rrd]^{1_{\Ext(\Cs)}} &&\\
\ExtS(\Cs) \ar@{}[rru]|(0.3){\Uparrow^{1_K}} \ar[rr]_{K} \ar[u]^-{K} & &\Ext(\Cs)
}
\]
and it is preserved by $I_1$ (see \cite{MacLane}), that is $I_1$ is the left pointwise Kan extension of the functor $I_1 \circ K\cong \tilde K \circ T_1$ along $K$:
\[
\xymatrix{
\Ext(\Cs) \ar@{.>}[rrd]^{1_{\Ext(\Cs)}} \ar@/^4ex/[rrrrd]^-{I_1=I_1 \circ 1_{\Ext(\Cs)}}&&\\
\ExtS(\Cs) \ar@{}[rru]|(0.3){\Uparrow^{1_K}} \ar[rr]_{K} \ar[u]^-{K} \ar@/_7ex/[rrrr]_-{\tilde K\circ T_1}& &\Ext(\Cs) \ar[rr]_-{I_1} \ar@{}[d]|-{\cong}  \ar@{}[u]|-{\Uparrow^{1_{I_1} * 1_K}}&& \NExt(\Cs).\\
&&}
\]
\end{proof}

Theorem 4.3 has the following corollary (a similar result was obtained independently by Montoli, Rodelo and Van der Linden in \cite{MRVdL4}*{Theorem 2.10}):

\begin{corollary} If $\Gamma=(\Cs,\Xs, I,H ,\eta,\epsilon, \Ec,\Fc)$ satisfies the assumptions of Theorem \ref{normalisationasKanextension} and if the coequaliser (let us write $\overline{f}$ for its codomain) of 
\[
\xymatrix{
T_1(p_1^f) \ar@<0.5ex>[rr]^-{T_1(\tau,\pi_1^f)} \ar@<-0.5ex>[rr]_-{T_1(p_2^f,\pi_2^f)}  && T_1(\pi_1^f)
}
\]
exists in $\NExt(\Cs)$ for every $f$ in $\Ext(\Cs)$, then the normalisation functor  $I_1$ exists and $I_1(f)=\overline{f}$.

\end{corollary}

\section*{Acknowledgements} We are grateful to Tim Van der Linden for helpful comments on a preliminary version of the paper.

\bibliography{tim}
\bibliographystyle{amsplain}

\end{document}